\documentclass[11pt]{amsart}
\usepackage{mathrsfs,latexsym,amsfonts,amssymb}
\usepackage{hyperref}
\newtheorem{theorem}{Theorem}[section]
\newtheorem{lemma}[theorem]{Lemma}
\newtheorem{corollary}[theorem]{Corollary}
\newtheorem{question}[theorem]{Question}
\newtheorem{remark}[theorem]{Remark}
\theoremstyle{definition}
\newtheorem{definition}[theorem]{Definition}
\newtheorem{proposition}[theorem]{Proposition}

\newtheorem{problem}[theorem]{Problem}

\begin{document}

\title[A note on hyperspaces by closed sets with Vietoris topology]
{A note on hyperspaces by closed sets with Vietoris topology}

\author{Chuan Liu}
\address{(Chuan Liu)Department of Mathematics,
Ohio University Zanesville Campus, Zanesville, OH 43701, USA}
\email{liuc1@ohio.edu}

 \author{Fucai Lin}
 \address{(Fucai Lin) 1. School of mathematics and statistics, Minnan Normal University, Zhangzhou 363000, P. R. China; 2. Fujian Key Laboratory of Granular Computing and Application, Minnan Normal University, Zhangzhou 363000, P. R. China}
 \email{linfucai@mnnu.edu.cn; linfucai2008@aliyun.com}

 \thanks{The second author is supported by the Key Program of the Natural Science Foundation of Fujian Province (No: 2020J02043), the NSFC (No. 11571158), the lab of Granular Computing, the Institute of Meteorological Big Data-Digital Fujian and Fujian Key Laboratory of Data Science and Statistics.}


\begin{abstract}
For a topological space $X$, let $CL(X)$ be the set of all non-empty closed subset of $X$, and denote the set $CL(X)$ with the Vietoris topology by $(CL(X),
\mathbb{V})$. In this paper, we mainly discuss the hyperspace $(CL(X),
\mathbb{V})$ when $X$ is an infinite countable discrete space. As an application, we first prove that the hyperspace with the Vietoris topology on an infinite countable discrete space contains a closed copy of $n$-th power of Sorgenfrey line for each $n\in\mathbb{N}$.
Then we investigate the tightness of the hyperspace $(CL(X),
\mathbb{V})$, and prove that the tightness of $(CL(X),
\mathbb{V})$ is equal to the set-tightness of $X$. Moreover, we extend some results about the generalized metric properties on the hyperspace $(CL(X),
\mathbb{V})$. Finally, we give a characterization of $X$ such that $(CL(X),
\mathbb{V})$ is a $\gamma$-space.

\end{abstract}

\maketitle

\section{Introduction}
It is well known that the topics of the hyperspace has been the
focus of much research, see \cite{GM2016, HPZ2003, HL1997, K1978, LSL2021, LLL2021, M1951, NP2015,
N1985, PS2017, TLL2018}. There are many results on the hyperspace
$CL(X)$ of closed subsets of a topological space equipped with
various topologies. In this paper, we endow $CL(X)$ with the Vietoris topology $\mathbb{V}$, or the
so-called finite topology, the base of which consists of all subsets
of the following form: $$\langle U_1, ..., U_k\rangle =\{K\in CL(X): K\subset
{\bigcup}_{i=1}^k U_i\ \mbox{and}\ K\cap U_j\neq \emptyset, 1\leq
j\leq k\},$$where each $U_i$ is open in $X$ and $k\in\mathbb{N}$. We denote the hyperspace $CL(X)$ with Vietoris topology by $(CL(X),
\mathbb{V})$. In 1997, Hol\'{a} and Levi in \cite{HL1997} gave a characterization of the first countability of $(CL(X),
\mathbb{V})$; in 2002, Hol\'{a}, Pelant and Zsilinszky in \cite{HPZ2003} proved that $(CL(X),
\mathbb{V})$ is developable iff $(CL(X),
\mathbb{V})$ is Moore iff $(CL(X),
\mathbb{V})$ is metrizable iff $(CL(X),
\mathbb{V})$ has a $\sigma$-discrete network iff $X$ is compact and metrizable. So it is natural for us to
consider the following two problems:

\begin{problem}\label{p6}
Let $\mathcal{C}$ be a proper subclass of the class of first-countable spaces, and let $\mathcal{P}$ be a topological property. If $(CL(X),
\mathbb{V})\in\mathcal{C}$, does $X$ have the property $\mathcal{P}$?
\end{problem}

\begin{problem}\label{p7}
Let $\mathcal{C}$ be a class of generalized metrizable spaces. If $(CL(X),
\mathbb{V})\in\mathcal{C}$, is $X$ compact and metrizable?
\end{problem}

The paper is organized as follows. In Section 2, we introduce the
necessary notation and terminology which are used in
the paper. In Section 3, we mainly discuss the hyperspace $(CL(D(\omega)),
\mathbb{V})$, and prove that $(CL(D(\omega)), \mathbb{V})$ contains a closed copy of $\mathbb{S}^{n}$ for each $n\in\mathbb{N}$, where $\mathbb{S}$ is the Sorgenfrey line. In Section 4,  we prove that the tightness of $(CL(X),
\mathbb{V})$ is equal to the set-tightness of $X$; moreover, we give a characterization of $(CL(X),
\mathbb{V})$ which is Fr\'echet-Urysohn. In section 5, we give some answers to Problems~\ref{p6} and~\ref{p7} respectively. In particular, we prove that $(CL(X), \mathbb{V})$ is quasi-developable iff $(CL(X), \mathbb{V})$ is a semi-stratifiable space iff $(CL(X), \mathbb{V})$ is symmetrizable iff $(CL(X), \mathbb{V})$ is a $D_{1}$-space iff $X$ is compact and metrizable; moreover, we prove that $(CL(X), \mathbb{V})$ is a $\gamma$-space iff $X$  is a
separable metrizable space and $S(X)$ is compact, where $S(X)$ is the set of all non-isolated points of $X$.

\smallskip
\section{Preliminaries}
In this paper, the base space $X$ is always supposed to be regular. Let $\mathbb{N}$ and $\omega$ denote the sets of all positive integers and all non-negative integers, respectively. Let $\mathbb{S}$ be the real line endowed with half open interval topology, that is, Sorgenfrey line. For a space $X$, $S(X)$ is the set of all non-isolated points of $X$. For undefined notations and terminologies, the reader may refer to \cite{E1989},  \cite{G1984} and \cite{M1951}.

Let $X$ be a topological space and $A \subseteq X$ be a subset of $X$.
The \emph{closure} of $A$ in $X$ is denoted by $\overline{A}$. A subset $P$ of $X$ is called a
\emph{sequential neighborhood} of $x \in X$, if each
sequence converging to $x$ is eventually in $P$. A subset $U$ of
$X$ is called \emph{sequentially open} if $U$ is a sequential neighborhood of
each of its points. A subset $F$ of
$X$ is called \emph{sequentially closed} if $X\setminus F$ is sequentially open. The space $X$ is called a \emph{sequential space} if each
sequentially open subset of $X$ is open. The space $X$ is said to be {\it Fr\'{e}chet-Urysohn} if, for
each $x\in \overline{A}\subset X$, there exists a sequence
$\{x_{n}\}$ in $A$ such that $\{x_{n}\}$ converges to $x$.

\begin{definition}
Let $\mathscr P$ be a cover of a space $X$ such that (i) $\mathscr P
=\bigcup_{x\in X}\mathscr{P}_{x}$; (ii) for each $x\in X$, if
$U,V\in\mathscr{P}_{x}$, then $W\subseteq U\cap V$ for some $W\in
\mathscr{P}_{x}$; (iii) $x\in \bigcap\mathscr{P}_{x}$ for each $x\in
X$; and (iv) for each point $x\in X$ and each open neighborhood $U$
of $x$ there is some $P\in\mathscr P_x$ such that $x\in P \subseteq
U$.

\smallskip
$\bullet$ The family $\mathscr{P}$ is called a {\it weak base} for
$X$ if, for every $G\subset X$, the set $G$ must be open in $X$
whenever for each $x\in G$ there exists $P\in \mathscr{P}_{x}$ such
that $P\subset G$, and $X$ is {\it weakly first-countable} if $X$ has a
weak base $\mathscr P$ and $\mathscr{P}_{x}$ is countable for each
$x\in X$.
\end{definition}

\begin{definition}
Let $\mathscr P$ be a family of subsets of a space $X$. The family $\mathscr P$ is called a {\it $k$-network} if
for every compact subset $K$ of $X$ and an arbitrary open set $U$
containing $K$ in $X$ there is a finite subfamily $\mathscr
{P}^{\prime}\subseteq\mathscr {P}$ such that $K\subseteq
\bigcup\mathscr {P}^{\prime}\subseteq U$.
\end{definition}

A space $X$ is said to be \emph{La\v{s}nev} if it is the
continuous closed image of some metric space. The following La\v{s}nev space in Definition~\ref{d00} plays an
important role in the study of the generalized metric theory.

\begin{definition}\label{d00}
Let $\kappa$ be an infinite cardinal.
  For each $\alpha\in\kappa$, let $T_{\alpha}$ be a sequence converging to
  $x_{\alpha}\not\in T_{\alpha}$. Let $T=\bigoplus_{\alpha\in\kappa}(T_{\alpha}\cup\{x_{\alpha}\})$ be the topological sum of $\{T_{\alpha}
  \cup \{x_{\alpha}\}: \alpha\in\kappa\}$. Then
  $S_{\kappa}=\{x\}  \cup \bigcup_{\alpha\in\kappa}T_{\alpha}$
  is the quotient space obtained from $T$ by
  identifying all the points $x_{\alpha}\in T$ to the point $x$. The space $S_{\kappa}$ is called a {\it sequential fan}.
\end{definition}

The following space is not a La\v{s}nev space.

\begin{definition}\label{d01}
A space $X$ is called an \emph{ $S_{2}$}-{space} ({\it Arens'
space})  if
$$X=\{\infty\}\cup \{x_{n}: n\in \mathbb{N}\}\cup\{x_{n, m}: m, n\in
\omega\}$$ and the topology is defined as follows: Each $x_{n, m}$
is isolated; a basic neighborhood of $x_{n}$ is
$\{x_{n}\}\cup\{x_{n, m}: m>k\}$, where $k\in\omega$; a basic
neighborhood of $\infty$ is $$\{\infty\}\cup (\bigcup\{V_{n}:
n>k\})\ \mbox{for some}\ k\in \omega,$$ where $V_{n}$ is a
neighborhood of $x_{n}$ for each $n\in\omega$.
\end{definition}

Given a topological space $X$, we define its {\it hyperspace} as the
following set:

$$CL(X)=\{H: H\ \mbox{is non-empty, closed in}\ X\}.$$

We endow $CL(X)$ with {\it Vietoris topology}
defined as the topology generated by the following family $$\{\langle U_1, \ldots,
U_k\rangle: U_1, \ldots, U_k\ \mbox{are open subsets of}\ X, k\in
\mathbb{N}\},$$ where  $\langle U_1, ..., U_k\rangle =\{H\in
CL(X): H\subset {\bigcup}_{i=1}^k U_i$ and $H\cap U_j\neq \emptyset,
1\leq j\leq k\}$. We denote this hyperspace with Vietoris topology by $(CL(X),
\mathbb{V})$.

If $U$ is a subset of $X$, then $$U^{-}=\{H\in CL(X): H\cap U\neq
\emptyset\}$$ and $$U^{+}=\{H\in CL(X): H\subset U\}.$$ Sometimes, we denote $U^{-}$ by $U^{-X}$ in order to prevent the confusion.

Let $X$ be a space. The {\it closed set character} (resp. {\it compact set character}) of $X$ is the minimal cardinal $\tau\geq\omega$ such that for each closed (resp. compact) set $A$ of $X$ the cardinal of the character of $A$ in $X$ is at most $\tau$. The closed set character (resp. compact set character) of $X$ is denoted by $cl\chi(X)$ (resp. $co\chi(X)$). If $cl\chi(X)=\omega$, then $X$ is called a {\it $D_1$-space} \cite{A1966} if $\sup\{\chi(H): H\in CL(X)\}\leq \omega$; if $co\chi(X)=\omega$, then $X$ is called a {\it $D_0$-space} \cite{S1975} if $\sup\{\chi(H): H\ \mbox{is compact in}\ X\}\leq \omega$. Clearly, each $D_1$-space is a $D_0$-space.

\smallskip
\section{The topological properties of hyperspace on an infinite countable discrete space}
In this section, we mainly discuss the topological properties of hyperspace on an infinite countable discrete space. First, we recall a concept.

A proper subset $C$ of the rational number $\mathbb{Q}$ is called a {\it cut} if $C$ has no largest
element and $(-\infty, p]\cap \mathbb{Q}\subset C$ for each $p\in C$.
If $C, D$ are cuts and $C$ is a proper subset of $D$, then denoted by
$C<D$.

In this paper, we always denote any countable infinite discrete space by $D(\omega)$.
The following lemma is a simple modification of \cite[Theorem 4.11]{HP2002}.

\begin{lemma}\label{l1}
The hyperspace $(CL(D(\omega)),
\mathbb{V})$ contains a closed copy of Sorgenfrey line $\mathbb{S}$.
\end{lemma}

\begin{proof}
Let $\mathbb{Q}$ be the set of rational number with the discrete
topology; then $D(\omega)$ is homeomorphic to $\mathbb{Q}$. Therefore, we may assume that $D(\omega)$ is $\mathbb{Q}$. Let $$\mathbb{X}=\{C\in CL(D(\omega)): C\ \mbox{is a cut}\}.$$ It was proved
that the subspace $\mathbb{X}$ of $(CL(D(\omega)), \mathbb{V})$ is
homeomorphic to the Sorgenfrey line by \cite[Theorem 4.11]{HP2002}.
Now we only rove that $\mathbb{X}$ is closed in $(CL(D(\omega)), \mathbb{V})$. Take any $C\in CL(D(\omega))\setminus \mathbb{X}$; then $C$ is not a cut. Hence $C$ has a largest element or there exist $p\in C$ such that $((-\infty, p]\cap
\mathbb{Q})\setminus C\neq \emptyset$. In order to find an open neighborhood $\widehat{U}$ of $C$ in $CL(D(\omega))$ such that $\widehat{U}\cap \mathbb{X}=\emptyset$, we divide the proof into the following two cases.

\smallskip
{\bf Case 1}: $C$ has a largest element $p$.

\smallskip
Then $p\in C$ such that $r\leq p$ for any $r\in C$. Clearly, $\langle C, \{p\}
\rangle$ is an open neighborhood of $C$ in $(CL(D(\omega)),
\mathbb{V})$, hence it easily follows that $\langle C, \{p\} \rangle\cap
\mathbb{X}=\emptyset$. Now put $\widehat{U}=\langle C, \{p\} \rangle$, as desired.

\smallskip
{\bf Case 2}: There exist $p\in C$ such that $((-\infty, p]\cap
\mathbb{Q})\setminus C\neq \emptyset$.

\smallskip
Pick any $q\in ((-\infty, p]\cap \mathbb{Q})\setminus C$; then $q<p$. Clearly, $\langle
C, \{p\}\rangle$ is an open neighborhood of $C$. We claim that $\langle C,
\{p\}\rangle\cap \mathbb{X}=\emptyset$. Indeed, if not, there exists a cut $D\in
\mathbb{X}$ such that $p\in D$ and $D\subset C$, then $q\in D$ since $q<p$
and $D$ is a cut. This is a contradiction since $q\notin C$. Now put $\widehat{U}=\langle C, \{p\} \rangle$, as desired.

\smallskip
Therefore, it follow from Cases 1 and 2 that $\mathbb{X}$ is closed in $(CL(D(\omega)), \mathbb{V})$.
\end{proof}

\begin{proposition}\label{pro}
Let $X$ be a space and $X=\bigoplus_{i\in\mathbb{N}}X_{i}$, where $X_{i}\cap X_{j}=\emptyset$ for any distinct $i$ and $j$. Then the box product $\prod_{i\in\mathbb{N}}(CL(X_{i}),
\mathbb{V})$ is homeomorphic to a closed
subspace of $(CL(X),\mathbb{V})$.
\end{proposition}

\begin{proof}
Let $$\mathbb{X}'=\{H\in CL(X): H\cap X_{i}\neq \emptyset, i\in\mathbb{N}\}.$$ We claim that $\mathbb{X}'$ is a closed subspace of
$(CL(X),\mathbb{V})$. Indeed, take any $K\in CL(X)\setminus\mathbb{X}'$; then
$K\cap X_{i}=\emptyset$ for some $i\in\mathbb{N}$. Put $Y=\bigcup_{j\in\mathbb{N}\setminus\{i\}}X_{j}$. Then $Y^{+}$ is a
neighborhood of $K$ and $Y^{+}\cap \mathbb{X}'=\emptyset$.  Now we prove that the box product $\prod_{i\in\mathbb{N}}(CL(X_{i}),
\mathbb{V})$ is
homeomorphic to $\mathbb{X}'$.

Indeed, define the mapping $f: \prod_{i\in\mathbb{N}}(CL(X_{i}),
\mathbb{V})\to \mathbb{X}'$ by $f(\prod_{i\in\mathbb{N}}C_{i})=\bigcup_{i\in\mathbb{N}}C_{i}$ for any $\prod_{i\in\mathbb{N}}C_{i}\in \prod_{i\in\mathbb{N}}(CL(X_{i}),
\mathbb{V})$. Clearly, $f$ is a bijection. Next it suffices to prove that $f$ is an open continuous mapping.

\smallskip
(1) The mapping $f$ is continuous.

\smallskip
Take any nonempty open subset $V$ of $X$. Then there exists a subset $A\subset \mathbb{N}$ such that $V\cap X_{n}\neq\emptyset$ for each $n\in A$ and $V\cap X_{m}\neq\emptyset$ for each $m\in \mathbb{N}\setminus\setminus A$. Then $$f^{-1}(V^{-X}\cap \mathbb{X}')=\bigcup_{i\in A}\left((V\cap X_{i})^{-X_{i}}\times \prod_{j\in\mathbb{N}\setminus\{i\}}X_{j}^{+}\right),$$ and then $$f^{-1}(V^{+})=\prod_{i\in\mathbb{N}}(X_{i}\cap V)^{+}$$ if $A=\mathbb{N}$. Hence $f$ is
continuous.

\smallskip
(2) The mapping $f$ is open.

\smallskip
Let $V_i\subset X_{i}$ be a nonempty open subset of
$X_{i}$ for each $i\in\mathbb{N}$. For any subset $B\subset \mathbb{N}$, we have
$$f(\prod_{i\in B}V_{i}^{-X_{i}}\times \prod_{j\in \mathbb{N}\setminus B}V_{j}^{+})=\bigcap_{i\in B}V_{i}^{-X}\cap (\bigcup_{i\in B}X_{i}\cup\bigcup_{j\in \mathbb{N}\setminus B}V_{j})^{+}\cap \mathbb{X}'.$$

Therefore, $f$ is a homeomorphism.
\end{proof}

By Proposition~\ref{pro}, we have the following theorem.

\begin{theorem}\label{tt}
The hyperspace $(CL(D(\omega)),
\mathbb{V})$ contains a closed copy of the box product $\prod_{n\in\mathbb{N}}\mathbb{S}_{n}$, where each $\mathbb{S}_{n}$ is homeomorphic the Sorgenfrey line $\mathbb{S}$.
\end{theorem}

\begin{proof}
We can write $D(\omega)=\bigcup_{i\in\mathbb{N}}E_{i}$ such that each $E_{i}$ is infinite and $E_i\cap E_j=\emptyset$ for distinct $i$ and $j$. From Proposition~\ref{pro}, it follows that the box product $\prod_{i\in\mathbb{N}}(CL(E_i),
\mathbb{V})$ is homeomorphic to a closed
subspace of $(CL(D(\omega)), \mathbb{V})$. By Lemma ~\ref{l1},
each $(CL(E_{i}), \mathbb{V})$ contains a closed copy
of Sorgenfrey line, hence $(CL(D(\omega)), \mathbb{V})$
contains a closed copy of the box product $\prod_{n\in\mathbb{N}}S_{n}$.
\end{proof}

From Theorem~\ref{tt}, we easily see the following corollary.

\begin{corollary}\label{c6}
The hyperspace $(CL(D(\omega)), \mathbb{V})$ contains a closed copy of $\mathbb{S}^{n}$ for each $n\in\mathbb{N}$.
\end{corollary}

\begin{remark}
It is well known that Sorgenfrey line $\mathbb{S}$ is a non-metrizable space which is hereditarily
Lindel\"of, hereditarily separable,
first-countable, perfect\footnote{A space $X$ is called {\it perfect} if every closed subset of $X$ is a
$G_\delta$-set.} and non-developable; moreover, it has the Baire property and a regular
$G_\delta$-diagonal. However, the square of Sorgenfrey line is not normal. Therefore, $(CL(D(\omega)), \mathbb{V})$ is not normal. Further, we have the following proposition.
\end{remark}

\begin{proposition}\label{p1}
The Sorgenfrey line $\mathbb{S}$ does not belong to any one of the following classes of spaces.
\begin{enumerate}

\smallskip
\item $\beta$-spaces\footnote{A space $(X, \tau)$ is called a {\it $\beta$-space} if there exists a function $g:
\mathbb{N}\times X\to \tau$ such that (i)
for any $x\in X$, we have $g(n+1, x)\subset g(n, x)$ for any $n\in\mathbb{N}$, (ii) for any $x\in X$ and sequence $\{x_{n}\}$ in $X$, if $x\in g(n, x_{n})$ for each $n\in\mathbb{N}$, then $\{x_{n}\}$ has an accumulation point in $X$};

\smallskip
\item spaces with a point-countable $k$-network;

\smallskip
\item spaces with a BCO\footnote{A space $X$ is said to have a {\it base of countable order} if there is a sequence
$\{\mathcal{B}_n\}$ of bases for $X$ such that: Whenever $x\in
b_n\in\mathcal{B}_n$ and $\{b_n\}$ is decreasing,
then $\{b_n: n\in \omega\}$ is a base at $x$. We use `BCO' to abbreviate
`base of countable order'.};

\smallskip
\item $p$-spaces\footnote{A regular space $X$ is called a {\it $p$-space} if there is a sequence $\{\mathscr{U}_{n}\}$ of families of open sets in $\beta X$ such that (1) each $\mathscr{U}_{n}$ covers $X$; (2) for each $x\in X$, $\bigcap_{n\in\mathbb{N}}\mbox{st}(x, \mathscr{U}_{n})\subset X$. If we also have (3) for each $x\in X$, $\bigcap_{n\in\mathbb{N}}\mbox{st}(x, \mathscr{U}_{n})=\bigcap_{n\in\mathbb{N}}\overline{\mbox{st}(x, \mathscr{U}_{n})}$, then $X$ is called a {\it strict $p$-space}.};

\smallskip
\item symmetrizable\footnote{A function $d: X\times X\rightarrow \mathbb{R}^{+}$ is called a {\it symmetric} on a set $X$ if for each $x, y\in X$, we have (1) $d(x, y)=0$ if and only if $x=y$ and (2) $d(x, y)=d(y, x)$. A space $(X, \tau)$ is called {\it symmetrizable} if there exists a symmetric $d$ on $X$ such that the topology $\tau$ given on $X$ is generated by the symmetric $d$, that is, a subset $U\in\tau$ if and only if for every $x\in U$, there is $\varepsilon>0$ such that $B(x, \varepsilon)\subset U$.}

\smallskip
\item quasi-developable spaces\footnote{A space $(X, \tau)$ is called {\it quasi-developable} if there exists a sequence $\{\mathscr{U}_{n}\}$ of families consisting of open sets in $X$ such that for each $x\in U\in\tau$ there exists $n\in\mathbb{N}$ such that $x\in\mbox{st}(x, \mathscr{U}_{n})\subset U$.};

\smallskip
\item $D_1$-spaces.
\end{enumerate}

Therefore, $(CL(D(\omega)), \mathbb{V})$ does not belong to any one of the classes of spaces (1)-(7).
\end{proposition}

\begin{proof}
(1) If the Sorgenfrey line $\mathbb{S}$ is a $\beta$-space, then it is a
Moore space\footnote{A space $(X, \tau)$ is called {\it developable} if there exists a sequence $\{\mathscr{U}_{n}\}$ of families of open covers of $X$ such that, for each $x\in X$,  $\{\mbox{st}(x, \mathscr{U}_{n})\}$ is an open neighborhood base of $x$ in $X$. A regular developable space is called a {\it Moore space}.}, hence $\mathbb{S}$ is metrizable since a
paratopological group which is a $\beta$-space is developable. This
is a contradiction.

\smallskip
(2) If the Sorgenfrey line $\mathbb{S}$ has a point-countable
$k$-network, then it has a point-countable base \cite{GMT1984}. Since a
separable space with a point-countable base is metrizable
\cite[Theorem 7.2]{G1984}, it follows that $\mathbb{S}$ is metrizable, this is
a contradiction.

\smallskip
(3) If the Sorgenfrey line $\mathbb{S}$ has a BCO, then it follows from \cite[Theorem
6.6]{G1984} that it is developable, hence it is metrizable. This is a
contradiction.

\smallskip
(4) If the Sorgenfrey line $\mathbb{S}$ is a $p$-space, then it is a
Lindel\"of $p$-space with a $G_\delta$-diagonal\footnote{A space $X$ is said to have a {\it $G_{\delta}$-diagonal} if, there is a sequence $\{\mathscr{U}_{n}\}$ of open covers of $X$, such that, for each $x\in X$, $\{x\}=\bigcap_{n\in\mathbb{N}}\mbox{st}(x, \mathscr{U}_{n})$.}, hence $\mathbb{S}$
is metrizable by \cite[Corollary 3.20, 3.4]{G1984}. This is a
contradiction.

\smallskip
(5) If the Sorgenfrey line $\mathbb{S}$ is symmetrizable, then it is a
semi-stratifiable\footnote{A space $(X, \tau)$ is called a {\it semi-stratifiable} if, there exists a function $F: \mathbb{N}\times \tau\rightarrow \tau^{c}$ satisfying the following conditions: (1) $U\in\tau\Rightarrow U=\bigcup_{n\in\mathbb{N}}F(n, U)$; (2) $V\subset U\Rightarrow F(n, V)\subset F(n, U)$.} space by \cite[Theorem 9.6]{G1984} and \cite[Theorem 9.8]{G1984}, hence a
$\beta$-space \cite[Page 475]{G1984}, this is a contradiction to
(1).

\smallskip
(6) If the Sorgenfrey line $\mathbb{S}$ is quasi-developable, then it is
developable by \cite[Theorem 8.6]{G1984}, this is a contradiction.

\smallskip
(7) If the Sorgenfrey line $\mathbb{S}$ is a $D_1$-space, then $\mathbb{S}$ is
metrizable \cite[Theorem 7(4)]{DL1995} since
$\mathbb{S}$ has a $G_\delta$-diagonal, this is a contradiction.
\end{proof}

\begin{theorem}\label{theo}
The hyperspace $(CL(D(\omega)), \mathbb{V})$ is non-archimedean quasi-metrizable; thus it is quasi-metrizable.
\end{theorem}

\begin{proof}
Let $D(\omega)=\{r_{n}: n\in\mathbb{N}\}$ endowed with a discrete topology $\tau$. Now we define a $g$-function from $\mathbb{N}\times D(\omega)\rightarrow \tau$ as follows (1) and (2):

\smallskip
(1) If $A\in CL(D(\omega))$ is a finite subset of $D(\omega)$, then we can enumerate $A$ as $\{r_{n_{1}}, \cdots, r_{n_{k}}\}$ for some $k\in\mathbb{N}$ such that $n_{1}<\ldots, <n_{k}$; then put $$G(m, A)=\langle\{r_{n_{1}}\}, \ldots, \{r_{n_{k}}\}\rangle$$ for each $m\in\mathbb{N}$.

\smallskip
(2) If $A\in CL(D(\omega))$ is an infinite subset of $D(\omega)$, then we can enumerate $A$ as $\{r_{n_{i}}\}_{i\in\mathbb{N}}$ such that $n_{i}<n_{i+1}$ for each $i\in\mathbb{N}$; then put $$G(m, A)=\langle\{r_{n_{1}}\}, \ldots, \{r_{n_{m}}\}, A\rangle$$ for each $m\in\mathbb{N}$.

Now it easily check the following two conditions hold.

\smallskip
(i) For each $A\in CL(D(\omega))$ the family $\{G(m, A)\}_{m\in\mathbb{N}}$ is a base at $A$ in $(CL(D(\omega)), \mathbb{V})$.

\smallskip
(ii) For each $A\in CL(D(\omega))$, if $B\in G(m, A)$, then $G(m, B)\subset G(m, A)$.

Therefore, it follows from \cite[Theorem 10.2]{G1984} that $(CL(D(\omega)), \mathbb{V})$ is non-archimedean quasi-metrizable.
\end{proof}

Let $C_{\omega}=\{\infty\}\cup\{x_{mn}: n, m\in\mathbb{N}\}$ be a countable infinite set. Endow $C_{\omega}$ with a topology $\upsilon$ as follows:

\smallskip
(1) Each single point set $\{x_{mn}\}$ is open in $C_{\omega}$;

\smallskip
(2) For each $k\in\mathbb{N}$, put $U_{k}=\{x_{mn}:  m\in\mathbb{N}, n\geq k+1\}\cup\{\infty\}$; the family $\{U_{k}\}$ is a base at the point $\infty$.

From Theorem~\ref{t-g}, it follows that $(CL(C_{\omega}), \mathbb{V})$ is a $\gamma$-space\footnote{A space $(X, \tau)$ is a {\it $\gamma$-space} if there exists a function $g:
\omega\times X \to \tau$ such that (i) $\{g(n, x): n\in \omega\}$ is
a base at $x$; (ii) for each $n\in \omega$ and $x\in X$, there
exists $m\in \omega$ such that $y\in g(m, x)$ implies $g(m,
y)\subset g(n, x)$. By \cite[Theorem 10.6(iii)]{G1984}, each $\gamma$-space is a $D_0$-space.}. However, the following question is still unknown for us.

\begin{question}\label{qq}
Is the hyperspace $(CL(C_{\omega}), \mathbb{V})$ quasi-metrizable?
\end{question}

\smallskip
\section{The characterizations of tightness in hyperspaces}
In this section, we mainly give a characterization of tightness in hyperspace; in particular, we give a characterization of hyperspace which is Fr\'echet-Urysohn. First, we recall and introduce some concepts.

The {\it tightness} of a space $X$ is the minimal cardinal $\tau\geq\omega$ such that if any $x$ is a cluster point of any
subset $A$ of $X$, then there is a subset $B$ of $A$ such that $|B|\leq\tau$ and $x$ is a cluster point of
$B$. The tightness of $X$ is denoted by $t(X)$.

\begin{definition}
Let $X$ be a space, $\mathcal{F}\subset CL(X)$ and $A\in CL(X)$.

\smallskip
(1) The set $A$ is called a {\it cluster set} of $\mathcal{F}$ in $X$ if for any finite open subsets
$\{V_i: i\leq k\}$ with $V_{i}\cap A\neq\emptyset$ ($i\leq k$) and any open neighborhood $U$ of $A$, there is a
$F\in \mathcal{F}$ such that $F\subset U$ and $F\cap V_i\neq
\emptyset$ for any $i\leq k$.

\smallskip
(2) The {\it set-tightness} of $X$ is the minimal cardinal $\tau\geq\omega$ such that if $A$ is a cluster set of any
$\mathcal{F}\subset CL(X)$, then there is a subfamily
$\mathcal{F}'\subset \mathcal{F}$ such that $|\mathcal{F}'|\leq\tau$ and $A$ is a cluster set of
$\mathcal{F}'$. The set-tightness of $X$ is denoted by $st(X)$.

\smallskip
(3) The sequence $\{A_j: j\in \mathbb{N}\}$ of $CL(X)$ is called {\it strongly converging} to $A$ in $X$ if
for any finite open subsets $\{V_i: i\leq k\}$ with $A\cap V_i\neq
\emptyset$ $  (j\leq k)$ and any open neighborhood $U$ of $A$, there
exists $N\in \mathbb{N}$ such that $A_j\subset U$ and $A_j\cap
V_i\neq \emptyset $ $ (i\leq k)$ whenever $j\geq N$.

\smallskip
(4) The space $X$ has {\it set-FU property} if whenever $A$ is a cluster set of
$\mathcal{F}\subset CL(X)$, there is a countable subfamily $\{A_j:
j\in \mathbb{N}\}$ such that $\{A_j: j\in \mathbb{N}\}$ strongly
converges to $A$ in $X$.
\end{definition}

From the definition of the set-FU property, it follows that if
elements of $\mathcal{F}$ and $A$ are all singleton, then $X$ is
Fr\'echet-Urysohn. Therefore, it easily see that there exists a countable set-tightness space $X$ such that $X$ is not set-FU property, such as Arens space $S_{2}$. Now we can use the concepts of set-FU property and set-tightness to characterize the Fr\'echet-Urysohn and tightness of $(CL(X), \mathbb{V})$ respectively. First, the following proposition gives a characterization of $X$ such that $t((CL(X), \mathbb{V}))\leq\tau$.

\begin{proposition}\label{p4}
Let $X$ be a space. Then $t((CL(X), \mathbb{V}))\leq\tau$ if and only if $st(X)\leq\tau$.
\end{proposition}

\begin{proof}
Sufficiency. Assume $t((CL(X), \mathbb{V}))\leq\tau$. Let $A$ be a cluster set of $\mathcal{F}\subset CL(X)$.
For any finite open subsets $\{V_i: i\leq k\}$ with $A\cap V_i\neq
\emptyset \ (i\leq k)$ and any open neighborhood $U$ of $A$, the set
$\langle V_1\cap U, ..., V_k\cap U, U \rangle$ is a neighborhood of
$A$ in $(CL(X), \mathbb{V})$. Since $t((CL(X), \mathbb{V}))\leq\tau$, there exists a subfamily $\mathcal{F}^{\prime}\subset \mathcal{F}$ such that $|\mathcal{F}^{\prime}|\leq\tau$ and $A\in\overline{\mathcal{F}^{\prime}}$ in $(CL(X), \mathbb{V})$. Since $A\in \langle V_1\cap
U, ... V_k\cap U, U \rangle$, there
exists $F\in \mathcal{F}^{\prime}$ such that $F\in \langle V_1\cap U, ...,
V_k\cap U, U \rangle$; then
$F\subset U, F\cap V_i\neq\emptyset$ for any $i\leq k$.
Therefore, $A$ be a cluster set of $\mathcal{F}^{\prime}$. Thus $st(X)\leq\tau$.

\smallskip
Necessity. Assume $st(X)\leq\tau$, and suppose that $A$ belongs to the closure of
$\mathcal{F}$ in $(CL(X), \mathbb{V})$, where $\mathcal{F}\subset CL(X)$. We claim that $A$ is a
cluster set of $\mathcal{F}$. Indeed, for any finite open subsets
$\{V_i: i\leq k\}$ and any open neighborhood $U$ of $A$, the set $\langle
V_1\cap U, ..., V_k\cap U, U \rangle$ is a neighborhood of $A$ in
$(CL(X), \mathbb{V})$. Then there exists $F\in \mathcal{F}$ such that
$F\in \langle V_1\cap U, ..., V_k\cap U, U \rangle$, which implies that
$F\cap V_i\neq \emptyset$ for $i\leq k$ and $F\subset U$. Hence $A$ is a cluster set of $\mathcal{F}$. Since $st(X)\leq\tau$, there is a subfamily $\mathcal{F}_{1}\subset \mathcal{F}$
such that $|\mathcal{F}_{1}|\leq\tau$ and $A$ is a cluster set of $\mathcal{F}_{1}$ in $X$. Finally it suffices to prove the following claim.

\smallskip
{\bf Claim:} $A\in \overline{\mathcal{F}_{1}}$ in $(CL(X), \mathbb{V})$.

\smallskip
Let  $\langle W_1, ..., W_m\rangle$ be a neighborhood of $A$ in
$(CL(X), \mathbb{V})$, and let $W=\cup\{W_i: i\leq m\}$. Then $A\cap
W_i\neq \emptyset$ for any $i\leq m$ and $A\subset W$. Since $A$ is a cluster set of $\mathcal{F}_{1}$, there exists $F\in \mathcal{F}_{1}$ such that $F\cap
W_i\neq \emptyset$ for any $i\leq m$ and $F\subset W$. Hence $F\in \langle W_1, ..., W_m\rangle$. Therefore, $A\in \overline{\mathcal{F}_{1}}$ in $(CL(X), \mathbb{V})$.
\end{proof}

\begin{corollary}
Let $X$ be a space. Then $(CL(X), \mathbb{V})$ is of countable tightness if and only if $X$ is of countable set-tightness.
\end{corollary}

\begin{proposition}\label{pro5}
Let $X$ be a (regular) space. Then we have the following statements:
\begin{enumerate}
\item $co\chi(X)\leq st(X)$;

\smallskip
\item If $X$ is a normal space, then $cl\chi(X)\leq st(X)$.
\end{enumerate}
\end{proposition}

\begin{proof}
We only prove (2), and the proof of (1) is similar. Let $st(X)=\tau$, let $A$ be an arbitrary closed subset of $X$, and let $\mathcal{B}_{A}=\{U_{\alpha}: \alpha\in I\}$ be an open neighborhood base at $A$ in $X$. Since $X$ is normal, it follows that $\mathcal{B}_{A}=\{\overline{U_{\alpha}}: \alpha\in I\}$ be a neighborhood base at $A$ in $X$, hence it easily check that $A$ is a cluster set of $\mathcal{B}_{A}$. Because $st(X)\leq\tau$, there exists a subfamily $\mathcal{B}_{A}^{\prime}=\{\overline{U_{\alpha}}: \alpha\in I_{1}\}$ of $\mathcal{B}_{A}$ such that $|I_{1}|\leq\tau$ and $A$ is a cluster set of $\mathcal{B}_{A}^{\prime}$. Therefore, for any open neighborhood $U$ of $A$ in $X$, there exists $\alpha\in I_{1}$ such that $A\subset \overline{U_{\alpha}}\subset U$. Hence $\mathcal{B}_{A}^{\prime}$ is a neighborhood of $A$ in $X$. Hence $cl\chi(X)\leq st(X)$.
\end{proof}

\begin{corollary}\label{c7}
If $X$ is a (regular) space with countable set-tightness, then
$X$ is a $D_0$-space; in particular, $X$ is a $D_{1}$-space if $X$ is normal.
\end{corollary}

By Proposition ~\ref{p4} and Corollary~\ref{c7}, we have the following corollary.

\begin{corollary}\label{c5}
If $X$ is a space and $(CL(X), \mathbb{V})$ has countable tightness,
then $X$ is a $D_0$-space; in particular, $X$ is a first-countable space.
\end{corollary}

By \cite[Proposition 3]{AB1996}, it is natural to pose the following question.

\begin{question}
Let $X$ be a space. If $(CL(X), \mathbb{V})$ has countable tightness, does then $(CL(X), \mathbb{V})$ contain a copy of $S_{\omega}$?
\end{question}

\begin{question}\label{q4}
Under what conditions of a space $X$, we have $t(X)=st(X)$.
\end{question}

The following proposition gives a partial answer to Question~\ref{q4}.

\begin{proposition}\label{p5}
Let $X$ be a normal space. Then $X$ has countable set-tightness if and only if $X$ has the following properties:
\begin{enumerate}
\item $X$ is perfectly normal;

\smallskip
\item the set $X\setminus S(X)$ is countable;

\smallskip
\item $S(X)$ is countably compact, hereditarily separable and $\chi(S(X), X)\leq\aleph_{0}$.
\end{enumerate}
\end{proposition}

\begin{proof}
By \cite[Corollary 1.8]{HL1997}, the necessity is obvious. Let $X$ have countable set-tightness. By Proposition~\ref{p4}, $(CL(X), \mathbb{V})$ has countable tightness. Since $X$ is normal, it follows from \cite[Proposition~2.6]{HP2002} that $(CL(X), \mathbb{V})$ is first-countable. Then the sufficiency holds by \cite[Corollary 1.8]{HL1997}.
\end{proof}

\begin{remark}
By Proposition~\ref{p5}, there exists a metrizable space $X$ such that $X$ is not countable set-tightness. Indeed, let $X$ be an arbitrary non-compact metrizable space such that any point of $X$ is not isolated. By Proposition~\ref{p5}, $X$ is not countable set-tightness.
\end{remark}

The gap between $D_1$-spaces and
$D_0$-spaces is large, see \cite{DL1995}. The following proposition
gives some relations between $D_0$-spaces and other generalized
metric spaces.

\begin{proposition}
Let $X$ be a developable
space or a space with a point-countable base. Then $X$ is a $D_0$-space.
\end{proposition}

\begin{proof}
Fix an arbitrary compact subset $K\subset X$.

\smallskip
(1) Assume that $X$ is a space with a point-countable base. Let $\mathcal{B}$ be a point-countable base of $X$. Then $K$ is
metrizable by \cite[Corollary 7.11(ii)]{G1984}. Let $D$ be a
countable dense subset of $K$, and put $\mathcal{B}'=\{B\in \mathcal{B}:
B\cap K\neq \emptyset\}$; then $|\mathcal{B}'|\leq \omega$. Let
$$\mathcal{B}''=\{\cup \mathcal{F}: \mathcal{F}\subset \mathcal{B}'\ \mbox{is a finite cover of}\ K\}.$$ We prove that $\mathcal{B}''$ is a
countable base of $K$. Indeed, if $K\subset U$ with $U$ open, then, for
any $x\in K$, pick $B_x\in \mathcal{B}^{\prime}$ such that $x\in B_x\subset U$.
Since $\{B_x: x\in K\}$ is an open cover of $K$, there exists $n\in
\mathbb{N}$ such that $\{B_{x_i}: i\leq n\}$
is a finite open cover of $K$, then $K\subset \bigcup_{i\leq n}B_{x_i}\subset U$
and $\bigcup_{i\leq n}B_{x_i}\in \mathcal{B}''$.

\smallskip
(2) Let $X$ be a developable space, and let $Y$ be the quotient
space by identifying $K$ to a point $z$ with the canonical map $f$. It is
easy to see that $f$ is a perfect map. Since developable spaces are
preserved by perfect maps, then $Y$ is developable. Let $\{U_n:
n\in\mathbb{N}\}$ be a countable local base at $z$,
and put $V_n=f^{-1}(U_n)$ for each $n\in\mathbb{N}$. Then $\{V_n: n\in \mathbb{N}\}$ is a
countable base of $K$. Hence $X$ is a $D_0$-space.
\end{proof}

From Theorems~\ref{t100} and~\ref{t-g} below, there exists a space $X$ such that
$(CL(X), \mathbb{V})$ is a $D_{0}$-space, but $(CL(X), \mathbb{V})$ is not a $D_{1}$-space. Indeed, let $X$ be the space of topological sum of a compact metrizable space $C$ and a countable infinite discrete space $D$, that is, $X=C\bigoplus D$. Then it follows that $(CL(X), \mathbb{V})$ is a $D_{0}$-space and not a $D_{1}$-space.

The following proposition gives a characterization of $X$ such that $(CL(X), \mathbb{V})$ is Fr\'echet-Urysohn, which could be proved by a similar proof of
Proposition ~\ref{p4}.

\begin{proposition}\label{p3}
Let $X$ be a space. Then $(CL(X), \mathbb{V})$ is Fr\'echet-Urysohn if and only if $X$
has set-FU property.
\end{proposition}

It is well known that a strongly Fr\'echet-Urysohn space is Fr\'echet-Urysohn, but not vice versa. It is natural to pose the following two questions. Clearly, if Question~\ref{q2} is positive, then Question~\ref{q3} is also positive.

\begin{question}\label{q3}
Let $X$ be a space. If $(CL(X), \mathbb{V})$ is Fr\'echet-Urysohn, is then $(CL(X), \mathbb{V})$ strongly Fr\'echet-Urysohn?
\end{question}

\begin{question}\label{q2}
Let $X$ be a space. If $(CL(X), \mathbb{V})$ contains a (closed) copy of $S_{\omega}$, does then $(CL(X), \mathbb{V})$ contain a (closed) copy of $S_{2}$?
\end{question}

\smallskip
\section{Some generalized metric properties on hyperspaces}
In this section, we mainly give the characterizations of some generalized metric properties on hyperspaces, such as semi-stratifiable spaces, quasi-developable spaces, $D_{1}$-spaces, symmetrizable spaces, $\gamma$-spaces, etc.

First, we prove the first main theorem in this section as follows, which gives a partial answer to Problem~\ref{p7}.

\begin{theorem}\label{t100}
Let $X$ be a space. Then the following statements are equivalent.
\begin{enumerate}
\smallskip
\item $(CL(X), \mathbb{V})$  is a semi-stratifiable space;

\smallskip
\item $(CL(X), \mathbb{V})$  is quasi-developable;

\smallskip
\item $(CL(X), \mathbb{V})$ is a $D_{1}$-space;

\smallskip
\item $(CL(X), \mathbb{V})$ is symmetrizable;

\smallskip
\item $X$ is a compact metrizable space.
\end{enumerate}
\end{theorem}

In order to give the proof, we give some technique lemmas and theorems.

\begin{lemma}\label{l2}
Let $\mathcal{P}$ be a topological property that is closed
hereditary, and let there exist $n\in\mathbb{N}$ such that $\mathbb{S}^{n}$ does not have the property $\mathcal{P}$. If
$(CL(X), \mathbb{V})$ has the property $\mathcal{P}$, then $X$ is
countably compact.
\end{lemma}

\begin{proof}
Suppose $X$ is not countably compact, then there exists a closed,
countable infinite discrete subset $D(\omega)\subset X$. Then $(CL(D(\omega)), \mathbb{V})$ is a closed
subspace of $(CL(X), \mathbb{V})$. By Corollary~\ref{c6}, $(CL(X),
\mathbb{V})$ contains a closed copy of $\mathbb{S}^{n}$ for each $n\in\mathbb{N}$, then $\mathbb{S}^{n}$
has the property $\mathcal{P}$, this is a contradiction. Hence $X$
is countably compact.
\end{proof}

Since all properties in Proposition ~\ref{p1} are closed hereditary,
it follows from Lemma ~\ref{l2} that we have the following theorem.

\begin{theorem}\label{t1}
If $(CL(X), \mathbb{V})$ belongs to any one of spaces in
Proposition~\ref{p1}, then $X$ is countably compact.
\end{theorem}

Since each strict $p$-space is a $\beta$-space \cite[page 475]{G1984}, it follows from Theorem~\ref{t1} that we have the following corollary.

\begin{corollary}
A space $X$ is compact if and only if $(CL(X), \mathbb{V})$ is a strict
$p$-space.
\end{corollary}

\begin{proof}
If $(CL(X), \mathbb{V})$ is a strict
$p$-space, then it follows from Theorem ~\ref{t1} that $X$ is countably compact. Since $X$ is
a strict $p$-space, $X$ is submetacompact, hence $X$ is
compact. If $X$ is compact, then it follows from \cite[Corollary 13]{CLP2002} that $(CL(X), \mathbb{V})$ is a strict
$p$-space.
\end{proof}

\begin{remark}
It is well known that $(CL(X), \mathbb{V})$ is locally
compact if and only if $X$ is compact if and only if $(CL(X),
\mathbb{V})$ is compact, see \cite[Corollary 13]{CLP2002}. Both locally
compact spaces and  strict $p$-spaces are $p$-spaces, it is natural
to ask the following question.
\end{remark}

\begin{question}
If $(CL(X), \mathbb{V})$ is a $p$-space, is then $X$ compact?
\end{question}

\begin{lemma}\label{l4}
Let $\mathcal{P}$ be a property that is closed hereditary, and let there exists some $n\in\mathbb{N}$ such that $\mathbb{S}^{n}$ doesn't have the property $\mathcal{P}$. Then a space $X$ is compact metrizable if and only if $(CL(X), \mathbb{V})$
is perfect and has property $\mathcal{P}$.
\end{lemma}

\begin{proof}
It suffices to prove the sufficiency. By Lemma ~\ref{l2}, $X$ is countably compact.  Next we prove that $X$ has
a $G_\delta$-diagonal. Since $X$ is a closed subset of $(CL(X),
\mathbb{V})$ (indeed, $X$ is the set $\{\{x\}: x\in X\}$), there exists a sequence $\{U_n: n\in \mathbb{N}\}$ of open
subsets of $(CL(X), \mathbb{V})$ such that $X=\bigcap_{n\in
\mathbb{N}}U_n$. Put $F_2(X)=\{B\subset X: |B|\leq 2\}$; then
$\{F_2(X)\cap U_n: n\in \mathbb{N}\}$ is a countable family of open
subsets of $F_2(X)$. Define $f_2: X\times X\to F_2(X)$ by $f_2(x,
y)=\{x, y\}$; then it is well known that $f_2$ is an open and closed
continuous mapping from $X^2$ to $F_2(X)$. Note that
$f^{-1}_2(X)=\{(x, x): x\in X\}=\Delta$ is the diagonal of $X$,
and that $\{f^{-1}_2(F_2(X)\cap U_n): n\in \mathbb{N}\}$ is a countable
family of open neighborhoods of $\Delta$ and $\bigcap_{n\in
\mathbb{N}}f^{-1}_2(U_n\cap F_2(X))=\Delta$, hence $X$ has a
$G_\delta$-diagonal; therefore, $X$ is compact metrizable by
\cite[Theorem 2.14]{G1984}.
\end{proof}

Now we can prove our first theorem.

{\bf Proof of Theorem~\ref{t100}.} Clearly, it suffices to prove that (1), (2), (3), (4) $\Rightarrow$ (5). By Lemma~\ref{l4}, we have (1) $\Rightarrow$ (5) and (3) $\Rightarrow$ (5).

(2) $\Rightarrow$ (5).  Assume that $(CL(X), \mathbb{V})$  is quasi-developable. Then, by Theorem~\ref{t1}, $X$ is  countably compact, thus it is metrizable by \cite[Theorem
8.5]{G1984} and \cite[Corollary 8.3(ii)]{G1984}.

(4) $\Rightarrow$ (5). Let $(CL(X), \mathbb{V})$ be symmetrizable; then $(CL(X), \mathbb{V})$ has countable tightness, hence from Theorem
~\ref{t1} and Corollary ~\ref{c5}, it follow that $X$ is first-countable and countably compact. Since a first-countable, symmetrizable space is
semi-stratifiable, $X$ has a $G_\delta$-diagonal. Hence $X$ is
metrizable by \cite[Theorem 2.14]{G1984}. Therefore, $X$ is compact metrizable. The proof is completed.

\smallskip
It was proved that if $(CL(X), \mathbb{V})$ is a $\sigma$-space
(i.e. a regular space with a $\sigma$-discrete network\footnote{A family $\mathscr{P}$ in a space $X$ is called a {\it network} for $X$ if, for each $x\in U$ with $U$ open in $X$, there exists $P\in\mathscr{P}$ such that $x\in P\subset U$.}), then $X$ is compact
metrizable by \cite[Theorem 4.14]{HP2002}. So it is natural to ask the
following question.

\begin{question}\label{q0}
If $(CL(X), \mathbb{V})$ has a $\sigma$-locally countable network,
is then $X$ compact metrizable?
\end{question}

We give a partial answer to Question~\ref{q0}. First, we give a lemma.

\begin{lemma}\label{l3}
If $X$ is a (regular) space having a $\sigma$-locally countable
network, then each singleton is a $G_\delta$-set.
\end{lemma}

\begin{proof}
Let $\mathcal{P}=\bigcup_{n\in \mathbb{N}}\mathcal{P}_n$ be a
$\sigma$-locally countable network of $X$, where each $\mathcal{P}_n$ is locally countable. Since $X$ is regular, we may
assume that each element of $\mathcal{P}$ is closed. Fix any $x\in X$.
For $n\in \mathbb{N}$, let $U_n$ be an open neighborhood of $x$ such
that $U_n$ intersects at most countably many elements of $\mathcal{P}_n$, and let $\mathcal{P}_n'=\{P\in\mathcal{P}_n: P\cap U_{n}\neq\emptyset\}$.
For each $n\in\mathbb{N}$, enumerate $\{P\in \mathcal{P}_n', x\notin
P\}$ as $\{P_{n, i}: i\in \mathbb{N}\}$, and let $V_{n, i}=X\setminus P_{n, i}$ for each $i\in \mathbb{N}$;
then $V_{n, i}$ is open and $x\in V_{n, i}$ for each $i\in \mathbb{N}$. Now it suffices to prove the following claim.

\smallskip
{\bf Claim:} $\{x\}=(\bigcap_{n\in\mathbb{N}}U_n)\cap (\bigcap_{n, i\in\mathbb{N}}V_{n, i})$.

\smallskip
Suppose not, then there exists $y\neq x$ such that $y\in(\bigcap_{n\in\mathbb{N}}U_n)\cap (\bigcap_{n, i\in\mathbb{N}}V_{n, i})$. Let
$V$ be an open neighborhood of $y$ with $x\notin V$. Pick $P\in
\mathcal{P}$ with $y\in P\subset V$. Then $P=P_{k, j}$ for some $k,
j\in \mathbb{N}$, and $y\notin V_{k, j}=X\setminus P_{k, j}$. This
is a contradiction.
\end{proof}

\begin{theorem}
(MA($\omega_1)+TOP$) A (regular) space $X$ is compact metrizable if
and only if $(CL(X), \mathbb{V})$ has a $\sigma$-locally countable
network.
\end{theorem}

\begin{proof}
It suffices to prove the sufficiency. From Theorem ~\ref{t1}, it follows that $X$ is countably compact.
By Lemma ~\ref{l3}, each singleton of $(CL(X), \mathbb{V})$ is a
$G_\delta$-set, then $X$ is hereditarily separable by \cite[Proposition~4.3]{HPZ2003}. Under
MA($\omega_1)$+TOP, $X$ is Lindel\"of. Since a Lindel\"of space with
a $\sigma$-locally countable network has a countable network,
$X$ is a countably compact space with a countable network, hence it
is compact metrizable by \cite[Corollary 4.7(ii)]{G1984}.
\end{proof}

If $X$ is a $k$-space\footnote{A space $X$ is called a {\it $k$-space} if, for each $A\subset X$, $A$ is closed in $X$ provide $K\cap A$ is closed for each compact subset $K$ of $X$.}, we have the following result.

\begin{theorem}\label{ttt}
Let $X$ be  a (regular) $k$-space. Then $X$ is compact metrizable if
and only if $(CL(X), \mathbb{V})$ has a point-countable $k$-network.
\end{theorem}

\begin{proof}
By Theorem ~\ref{t1}, $X$ is countably compact. Since $X$ is $k$-space
with a point-countable $k$-network, it follows that $X$ is compact metrizable
space \cite{GMT1984}.
\end{proof}

We don't know whether we can delete the condition ``regular $k$-space'' in Theorem~\ref{ttt}, hence we have the following question.

\begin{question}
Suppose $(CL(X), \mathbb{V})$ has a point-countable $k$-network, is
$X$ metrizable?
\end{question}

The following theorem give a characterization of $X$ such that $(CL(X), \mathbb{V})$ has a BCO under the assumption of MA($\omega_1)+TOP$.

\begin{theorem}
(MA($\omega_1)+TOP$) A (regular) space $X$ is compact metrizable if
and only if $(CL(X), \mathbb{V})$ has a BCO.
\end{theorem}

\begin{proof}
By Theorem ~\ref{t1}, $X$ is countably compact. Moreover, it is obvious that each singleton of
$(CL(X), \mathbb{V})$ is a $G_\delta$-set, then $X$ is hereditarily
separable by \cite[Proposition~4.3]{HPZ2003}, hence it is Lindel\"of under MA($\omega_1$)+TOP. A
Lindel\"of space having a BCO is metrizable by \cite[Theorem
6.6]{G1984}, therefore, $X$ is compact and metrizable.
\end{proof}

Next we prove the second main theorems in this section, see Theorem~\ref{c3}. First we give some concepts and

A family $\mathcal{B}$ of open subsets of a space $X$ is called an
{\it external $\pi$-base of a subset $A$} if whenever $A\cap U\neq
\emptyset$ with $U$ open in $X$, there is $B\in \mathcal{B}$ such
that $A\cap B\neq \emptyset$ and $A\cap B\subset U$. We denote $$e\pi
w(A)=\inf\{|\mathcal{B}|: \mathcal{B}\  \mbox{is an external}\ \pi\mbox{-base of}\ A\}.$$ If $\mathcal{B}$ is an external $\pi$-base of $A$, then it easily see that $\{B\cap A: B\in \mathcal{B}\}$ is a $\pi$-base of $A$.

It was proved that $$\chi(CL(X), \mathbb{V})= hd(X)\cdot \sup\{\chi(H,
X): H\in CL(X)\}$$ (\cite[Theorem 2.2(5)]{HP2002}). We describe this
result in terms of external $\pi$-base.

\begin{proposition}\label{t5}
For a space $X$, we have $$\chi(CL(X), \mathbb{V})=\sup\{\chi(H, X): H\in CL(X)\} \cdot
\sup\{e\pi w(H): H\in CL(X)\}.$$
\end{proposition}

\begin{proof}
 Suppose $\chi(CL(X), \mathbb{V})\leq\kappa$. Fix any $H\in CL(X)$, and let
$\{\widehat{U}_\alpha: \alpha< \kappa\}$  be a local base at $H$ in $(CL(X), \mathbb{V})$. We write $\widehat{U}_\alpha=\langle U_1(\alpha), ..., U_{k_\alpha}(\alpha)\rangle$
for any $\alpha<\kappa$, where each $k_{\alpha}\in\mathbb{N}$. Let $W_\alpha=\cup_{j\leq k_\alpha}U_j(\alpha)$ for each $\alpha$; then it
is easy to check that $\{W_\alpha: \alpha<\kappa\}$ is a local base at $H$ in
$X$. Therefore, $\sup\{\chi(H, X): H\in CL(X)\}<\kappa$. Next we prove that the family $\mathcal{B}=\{U_j(\alpha): \alpha<\kappa, j\leq k_\alpha\}$ is an
external $\pi$-base of $H$. Indeed, let $V$ be an open subset of $X$ with $V\cap H\neq \emptyset$; then
$\langle V, X\rangle$ is a neighborhood of $H$, hence there exists $\alpha<\kappa$ such that
such that $\widehat{U}_\alpha\subset
\langle V, X\rangle$, then it follows from \cite[Lemma 2.3.1]{M1951} that $V$ contains
$U_j(\alpha)$ for some $j\leq k_{\alpha}$. Therefore, $\mathcal{B}$ is an external $\pi$-base
of $H$, that is, $\sup\{e\pi w(H): H\in CL(X)\}\leq \kappa$.

\medskip
 Suppose $\sup\{\chi(H, X): H\in CL(X)\}\leq\kappa$ and $\sup\{e\pi w(H):
H\in CL(X)\}\leq\kappa$. Fix any $H\in CL(X)$, let $\mathcal{W}$ be an
external $\pi$-base of $H$ in $X$ with $|\mathcal{W}|<\kappa$, and
let $\mathcal{U}=\{U_\alpha: \alpha<\kappa\}$ be a local base at $H$ in $X$. We claim that
 $$\{\langle W_1\cap U, ... , W_r\cap U, U\rangle:
\{W_1, ..., W_r\}\in \mathcal{W}^{<\omega}, U\in \mathcal{U}\}$$ is a
local base at $H$ in $(CL(X), \mathbb{V})$.

Indeed, let $\langle V_1, ... , V_p\rangle$ be an open neighborhood of $H$
in $(CL(X), \mathbb{V})$; then $H\subset \bigcup_{j\leq p}V_j$ and
$H\cap V_j\neq \emptyset$ for each $j\leq p$. Pick $U'\in
\mathcal{U}$ such that $U'\subset \bigcup_{j\leq p}V_j$, and $W_j\in
\mathcal{W}$ such that $W_j\subset V_j$ for each $j\leq p$. Then $H\in
\langle W_1\cap U', ... , W_p\cap U', U'\rangle\subset \langle V_1,
... , V_p\rangle$.
\end{proof}

By Proposition~\ref{t5}, it easily see that the second main theorem holds, which gives a partial answer to Problem~\ref{p6}.

\begin{theorem}\label{c3}
Let $X$ be a space. Then $(CL(X), \mathbb{V})$ is first-countable if and only if $X$ is a
$D_1$-space and each closed subset of $X$ has countable external
$\pi$-base.
\end{theorem}

It is well known that each first-countable space is weakly first-countable. The next theorem shows that weak first-countability is equivalent to first-countability in $(CL(X), \mathbb{V})$.

\begin{theorem}\label{tttt}
Let $X$ be a (regular) space. Then $(CL(X), \mathbb{V})$ is first-countable if and only if $(CL(X),
\mathbb{V})$ is weakly first-countable.
\end{theorem}

\begin{proof}
Clearly, it suffices to prove the sufficiency. Assume that $(CL(X),
\mathbb{V})$ is weakly first-countable, so it has countable tightness. Then $X$
is first-countable by Corollary ~\ref{c5}. Moreover, we claim that $d(A)\leq\omega$ for each $A\in CL(X)$. Indeed, take any $A\in CL(X)$, and let $\mathcal{F}=\{C: C\subset A, |C|<\omega\}$. Then $A$ belongs to the closure of $\mathcal{F}$ in $(CL(X), \mathbb{V})$. In fact, for any open neighborhood $\langle U_{1}, \ldots, U_{n}\rangle$ of $A$ in $(CL(X), \mathbb{V})$, we have $U_{i}\cap A\neq$ for any $i\leq n$; hence pick an arbitrary $x_{i}\in U_{i}\cap A\neq$ for any $i\leq n$. Then $\{x_{1}, \ldots, x_{n}\}\in \langle U_{1}, \ldots, U_{n}\rangle\cap \mathcal{F}\neq\emptyset$. Therefore, $A$ belongs to the closure of $\mathcal{F}$ in $(CL(X), \mathbb{V})$. Since $(CL(X), \mathbb{V})$ has a countable tightness, there exists a countable subset $\mathcal{F}_{1}=\{C_{n}: n\in\mathbb{N}\}$ of $\mathcal{F}$ such that $A$ belongs to the closure of $\mathcal{F}_{1}$ in $(CL(X), \mathbb{V})$. Put $D=\bigcup\mathcal{F}_{1}$. Then the closure of $D$ in $X$ is just $A$. In fact, let $U$ be an arbitrary open subset of $X$ such that $U\cap A\neq\emptyset$; then $U^{-}$ is an open neighborhood of $A$ in $(CL(X), \mathbb{V})$, hence $U^{-}$ contains some element $F\in\mathcal{F}_{1}$, then $F\cap U\neq\emptyset$. Therefore, $D\cap U\neq\emptyset$. Thus $d(A)\leq\omega$.

Take any $A\in CL(X)$; then, by Proposition~\ref{t5}, it suffices to prove that $e\pi w(A)=\omega$ and $\chi(A, X)=\omega$.

\smallskip
(1) $e\pi w(A)=\omega$.

\smallskip
Let $D$ be a countable dense subset of $A$; for each $d\in D$, let
$\mathcal{B}_d$ be a countable base at $d$ in $X$. Then it is easy to check that
$\bigcup\{\mathcal{B}_d: d\in D\}$ is an external $\pi$-base of $A$,
hence $e\pi w(A)=\omega$.

\smallskip
(2) $\chi(A, X)=\omega$.

\smallskip
Let $\{\mathcal{U}_i: i\in \mathbb{N}\}$ be a countable weak base at
$A$, and let $U_i=\bigcup\mathcal{U}_i$ for each $i\in \mathbb{N}$.
obviously, $A\subset U_i$ for each $i\in \mathbb{N}$. We prove that $\{\mbox{int}(U_i):
i\in \mathbb{N}\}$ is a countable base at $A$ in $X$.

First, we prove that each $U_i$ is a sequential neighborhood of $A$ in
$X$. Indeed, let $x_n\to x\in A$ as $n\rightarrow\infty$, and let $A_n=A\cup\{x_n\}$ for each $n\in\mathbb{N}$; then $A_n\to A$ as $n\rightarrow\infty$ in
$(CL(X), \mathbb{V})$. Since $\mathcal{U}_i$ is a weak neighborhood
of $A$ in $(CL(X), \mathbb{V})$, there exists $k\in \mathbb{N}$
such that $A_n\in \mathcal{U}_i$ whenever $n>k$, it implies
$A_n\subset U_i$ for $n>k$, hence $\{x_n: n>N\}\subset U_i$.

Second, for any $A\subset U$ with $U$ open in $X$, the set $\langle
U\rangle$ is an open neighborhood of $A$ in $(CL(X), \mathbb{V})$, hence
there exists $\mathcal{U}_i$ such that $A\in \mathcal{U}_i\subset
\langle U\rangle$, then $A\subset U_i\subset U$.

Finally, we prove $A\subset \mbox{int}(U_i)$ for each $i\in \mathbb{N}$. Suppose not,
pick any $x\in A\setminus \mbox{int}(U_i)$. Since $X$ is first-countable, there is a
sequence $\{x_n: n\in\mathbb{N}\}\subset X\setminus U_i$ such that
$x_n\to x$ as $n\rightarrow\infty$, which is a contradiction because $U_i$ is a sequential
neighborhood of $A$.

Therefore, $\{\mbox{int}(U_i): i\in \mathbb{N}\}$ is a countable base at
$A$, i.e. $\chi(A, X)=\omega$.
\end{proof}

Finally we prove the third main theorem in this section (see Theorem~\ref{t-g}), which also gives a partial answer to Problem~\ref{p6}.
We denote the set of non-isolated points of a space $X$ as $S(X)$.

\begin{theorem}\label{t-g}
Let $X$ be a space. Then $(CL(X), \mathbb{V})$ is a $\gamma$-space if and only if $X$ is a
separable metrizable space and $S(X)$ is compact.
\end{theorem}

\begin{proof}
Necessity. Clearly, $(CL(X), \mathbb{V})$ is first-countable, then it follows from Corollary
~\ref{c3} that $X$ is a $D_1$-space; moreover, $X$ is a $\gamma$-space since
the property of $\gamma$-space is hereditary. Therefore, $X$ is metrizable
by \cite[Theorem 7(8)]{DL1995}, and $S(X)$ is also countably compact by
\cite[Theorem 1]{DL1995}, thus $S(X)$ is compact. Since $X$  has a
countable external $\pi$-base by Theorem ~\ref{c3}, it follows that $X$ is
separable.

Sufficiency. Assume that $X$ is a
separable metrizable space and $S(X)$ is compact, and Assume that $d$ is the metric on $X$. Let $X=I(X)\cup S(X)$, where $I(X)$ is the set of all
isolated points of $X$. Clearly, $I(X)$ is countable, and we write $I(X)=\{r_1, r_2, ......
\}$. Let
$\mathcal{C}'$ be a countable base of $X$, and let
$\mathcal{C}=\{C\in \mathcal{C}': C\cap S(X)\neq \emptyset\}$; then
$\mathcal{C}$ is an external base\footnote{A family $\mathcal{B}$ of
open subsets of a space $X$ is called an external base \cite[Page
467]{AT2008} of a set $Y\subset X$ if for every point $y\in Y$ and
every neighborhood $U$ of $y$ in $X$ there exists $V\in \mathcal{B}$
such that $y\in V\subset U$.} of $S(X)$, and we write
$\mathcal{C}=\{C_1, C_2, ..., C_n,...\}$.
For each $A\in CL(X)$ and $n\in
\mathbb{N}$, we define a function $G: \mathbb{N}\times CL(X)\to \tau$ as
follows, where $\tau$ is the topology of $(CL(X), \mathbb{V})$.

\smallskip
{\bf Case 1:} $A\subset I(X)$.

\smallskip
If $A=\{a_1, ..., a_k\}$ is finite, then
put $G(n, A)=\langle \{a_1\}, ..., \{a_k\}\rangle=\{A\}$ for each $n\in\mathbb{N}$. If $A$ is
infinite, let $A=\{a_1, a_2, ..., a_n, ...\}$ such $a_i=r_{p_i}$ and $p_{i}<p_{i+1}$ for each $i\in
\mathbb{N}$; then put $G(n, A)=\langle \{a_1\}, ..., \{a_n\}, A\rangle$ for each $n\in\mathbb{N}$. We verify that the family $\{G(n, A): n\in\mathbb{N}\}$ satisfies the conditions (i) and (ii) of the definition of $\gamma$-space.

\smallskip
(i) Let $\mathbb{U}=\langle U_1, .., U_m\rangle$
be an arbitrary open neighborhood of $A$. Pick
$a_{j_i}\in U_i$ for $i\leq m$, and let $k=\max\{j_i: i\leq m\}$; then
$A\in\langle \{a_1\}, ..., \{a_k\}, A\rangle=G(k, A)\subset \mathbb{U}$.
Hence $\{G(n, A): n\in \mathbb{N}\}$ is a local base at $A$.

\smallskip
(ii) For any $n\in\mathbb{N}$, let $B\in G(n+1, A)$; then $\{a_1, ..., a_n, a_{n+1}\}\subset
B\subset A$. If $B$ is finite, it is obvious that $G(n+1, B)=\{B\}\subset
G(n, A)$; if $B=\{b_1, b_2, ... \}$ is infinite, then $b_i=a_i$ for any
$i\leq n+1$, hence $$G(n+1, B)=\langle \{b_1\}, ... , \{b_{n+1}\}, B\rangle \subset
\langle \{a_1\}, ..., \{a_{n+1}\}, A\rangle=G(n+1, A)\subset G(n, A).$$

\smallskip
{\bf Case 2:} $A\setminus I(X)\neq \emptyset$.

\smallskip
Then $A=A_1\cup A_2$, where
$A_1=A\cap I(X), A_2=A\cap S(X)$. Clearly, $A_2$ is compact. For each $n\in\mathbb{N}$,
let $B_{1/n}(A_2)=\{x\in X, d(A_2, x)<1/n\}$, and put $\mathcal{D}=\{D\in
\mathcal{C}, D\cap A_2\neq \emptyset\}$. Then we write
$\mathcal{D}=\{D_1, D_2, ..., D_k, ...\}$ such that $D_i=C_{q_i}$ for
$i\in \mathbb{N}$ and $\{q_i: i\in\mathbb{N}\}$ is increasing. For each $n\in\mathbb{N}$, let
$V_n=B_{1/n}(A_2)\cup A_1$. If $A_1$ is finite, then we can write $A_1=\{a_1,
..., a_m\}$, then put $$G(n, A)=\langle D_1\cap V_n, ..., D_n\cap V_n, \{a_1\},
..., \{a_m\}, V_n\rangle;$$ if $A_1$ is infinite, then we can write $A_1=\{a_1,
..., a_n,...\}$, then for each $n\in\mathbb{N}$ put $$G(n, A)=\langle D_1\cap V_n, ..., D_n\cap
V_n, \{a_1\}, ..., \{a_n\}, V_n\rangle.$$

Now it suffices to prove $G(n, A)$ satisfies (i) and (ii) in the definition of
$\gamma$-space as $A_1$ is infinite; for the case that $A_1$ is
finite, we may use a similar way to prove it.

\smallskip
(i$^{\prime}$) Let $\mathbb{U}=\langle U_1, .., U_m\rangle$
be an arbitrary open neighborhood of $A$. Since $A_2\subset
\bigcup\{U_i: i\leq m\}$, there is $n'\in \mathbb{N}$ such that
$A_2\subset B_{1/n'}(A_2)\subset \bigcup\{U_i: i\leq m\}$. For each $i\leq m$, if $U_i\cap A_2\neq \emptyset$, then we can find
$D_{j_i}\in \mathcal{D}$ such that $D_{j_i}\cap A_2\neq \emptyset$ and
$D_{j_i}\subset U_i$. Let $n''=\max\{j_{i}: i\leq m\}$, and let
$m'=\max\{n', n''\}$. Then $$A\in G(m', A)=\langle D_1\cap V_{m'}, ...,
D_{m'}\cap V_{m'}, \{a_1\}, ..., \{a_{m'}\}, V_{m'}\rangle\subset \langle
U_1, ..., U_m\rangle$$ by \cite[Lemma 2.3.1]{M1951}. Hence $\{G(n, A):
n\in\mathbb{N}\}$ is a countable local base at $A$.

\smallskip
(ii$^{\prime}$) For any $n\in \mathbb{N}$, let $s=q_n$,
then $s>n$. We claim that, for any $B\in G(s, A)$, we have $G(s, B)\subset G(n, A)$.
Indeed, it is obvious that $B\cap (D_i\cap V_s)\neq \emptyset$ for each $i\leq s$ and $\{a_1, ...,
a_s\}\subset B\subset V_s \subset V_n$. Let $\mathcal{E}=\{C\in
\mathcal{C}, C\cap B\neq \emptyset\}$; then we write $\mathcal{E}=\{E_i:
i\in \mathbb{N}\}$ such that $E_i=C_{l_i}$ for $i\in \mathbb{N}$ and
$\{l_i\}$ is increasing.  Note that $B\cap D_i\neq \emptyset$ and $B\cap
E_j\neq \emptyset$ for any $i\leq n, j\leq s$, we can see that $\{D_1,
..., D_n\}\subset \{E_1, ..., E_s\}$. Therefore, it follows from \cite[Lemma 2.3.1]{M1951} that
\begin{eqnarray}
G(s, B)&=&\langle E_1\cap V_s,
..., E_s\cap V_s, \{a_1\}, ..., \{a_s\}, V_s\rangle\nonumber\\
&\subset&\langle D_1\cap
V_n, ..., D_2\cap V_n, \{a_1\}, ..., \{a_n\}, V_n\rangle\nonumber\\
&=&G(n, A).\nonumber
\end{eqnarray}

Therefore, $(CL(X), \mathbb{V})$ is a
$\gamma$-space.
\end{proof}

\begin{corollary}\label{ccc}
The following statements are equivalent for a space $X$.

\begin{enumerate}
\item $(CL(X), \mathbb{V})$ is a $\gamma$-space;

\item $(CL(X), \mathbb{V})$ is a wealy first-countable and submetrizable space;

\item $(CL(X), \mathbb{V})$ is weakly first-countable and has a
$G_\delta$-diagonal;

\item $X$ is a separable metrizable space and $S(X)$ is compact.
\end{enumerate}
\end{corollary}

\begin{proof}
(1) $\Longleftrightarrow$ (4) by Theorem ~\ref{t-g}. (2)
$\Rightarrow$ (3) is trivial.

(3) $\Rightarrow$ (4). By Theorems
~\ref{c3} and~\ref{tttt}, $X$ is a separable $D_1$-space with a $G_\delta$-diagonal, then $S(X)$ is countably compact by
\cite[Theorem 1]{DL1995}, hence $S(X)$ is compact metrizable. Thus $X$ is
metrizable by \cite[Theorem 7(8)]{DL1995}.

(4) $\Rightarrow$ (2). By \cite[Proposition 8(2)]{HPZ2003}, $(CL(X),
\mathbb{V})$ is submetrizable. Moreover, $X$ is also first-countable by
\cite[Theorem 2.3]{HP2002}.
\end{proof}

By Theorem~\ref{tttt} and Corollary~\ref{ccc}, we have the following corollary.

\begin{corollary}
Let $X$ be a (regular) space. Then $(CL(X), \mathbb{V})$ is a $\gamma$-space if and only if $(CL(X),
\mathbb{V})$ is weakly first-countable and has a $G_\delta$-diagonal.
\end{corollary}

The following theorem shows that the classes of $D_0$-spaces and $\gamma$-spaces are equivalent in $(CL(X), \mathbb{V})$ under the assumption of $MA + \neg CH$.

\begin{theorem}
($MA + \neg CH$) Let $X$ be a space. Then $(CL(X), \mathbb{V})$ is a $D_0$-space if and only
if $(CL(X), \mathbb{V})$ is a $\gamma$-space.
\end{theorem}

\begin{proof}
By \cite[Theorem 10.6 (iii)]{G1984}, every $\gamma$-space is a
$D_0$-space, so the necessity is done.

Sufficiency. Assume $(CL(X), \mathbb{V})$ is a $D_0$-space, then, by Theorem~\ref{t-g}, it suffices to prove that $X$ is a separable metrizable space and $S(X)$ is compact.

By Theorem~\ref{c3}, $X$ is a $D_1$-space and every closed subset
of $X$ has countable external $\pi$-base, hence $S(X)$ is countably
compact by \cite[Theorem 1]{DL1995} and $X$ is hereditarily
separable. Under $MA +\neg CH$, $X$ is strongly paracompact by
\cite[Theorem 5 (2)]{DL1995}, which implies that $S(X)$ is compact.
Moreover, $S(X)$ is a closed subset of $X$, then $(CL(S(X)), \mathbb{V})$ is a
closed subspace of $(CL(X), \mathbb{V})$. Since $S(X)$ is
compact, it follows that $(CL(S(X)), \mathbb{V})$ is a compact $D_0$-space. Then
$(CL(S(X)), \mathbb{V})$ is a $D_1$-space since every closed subset of
$(CL(S(X)), \mathbb{V})$ is compact. By Theorem~\ref{t100}, $S(X)$
is compact metrizable, hence $X$ is metrizable by \cite[Theorem 7
(2)]{DL1995}. Therefore, $(CL(X), \mathbb{V})$ is a $\gamma$-space
by Theorem ~\ref{t-g}.
\end{proof}

Since each quasi-metrizable space is a $\gamma$-space, we have the following conjecture.

\smallskip
{\bf Conjecture 1:} Let $X$ be a space. Then $(CL(X), \mathbb{V})$ is qasi-metrizable if and only if $X=C\oplus D$, where $C$ is a compact metrizable space and $D$ is a countable discrete space.

\begin{remark}
If Question~\ref{qq} is affirmative, then it is obvious that this Conjecture 1 does not hold. If Question~\ref{qq} is negative, then this Conjecture 1 holds. Indeed, assume that $(CL(X), \mathbb{V})$ is qasi-metrizable, then it follows from Theorem~\ref{t-g} that $X$ is a separable metrizable space and $S(X)$ is compact. Since $(CL(C_{\omega}), \mathbb{V})$ is not qasi-metrizable, it follows that $X$ is locally compact, which implies that $S(X)$ is open in $X$. Therefore, $X$ is the topological sum of a compact metrizable space and a countable discrete space. Moreover, from Proposition~\ref{pro} and Theorem~\ref{theo}, it follows that $(CL(X), \mathbb{V})$ is qasi-metrizable if $X$ is the topological sum of a compact metrizable space and a countable discrete space.
\end{remark}

From Theorem~\ref{theo}, we also have the following conjecture.

\smallskip
{\bf Conjecture 2:} Let $X$ be a space. Then $(CL(X), \mathbb{V})$ is qasi-metrizable if and only if $(CL(X), \mathbb{V})$ is non-archimedean qasi-metrizable.

\end{document}